\documentclass[12pt,leqno]{amsart}
\usepackage{enumerate,enumitem}

\usepackage{vmargin,amscd}

\usepackage{amsmath, url}
\usepackage{amsfonts}
\usepackage{amssymb}
\usepackage[matrix,arrow,curve]{xy}
\newcommand{\dia}[1]{\begin{array}{c}{\xymatrix@C-3pt@M+2pt@R-4pt{#1 }}\end{array}}
\usepackage{tikz}
\newcommand*{\DashedArrow}[1][]{\mathbin{\tikz [baseline=-0.25ex,-latex, dashed,#1] \draw [#1] (0pt,0.5ex) -- (1.3em,0.5ex);}}

\newtheorem{thm}{Theorem}[section]
\newtheorem{cor}[thm]{Corollary}

\newtheorem{prop}[thm]{Proposition}
\newtheorem{thmBr}[thm]{Brouwer Fixed Point Theorem}
\newtheorem{thmBIDT}[thm]{Brouwer Invariance of Domain Theorem}
\theoremstyle{definition}
\newtheorem{defn}[thm]{Definition}

\newtheorem{exe}[thm]{Example}
\newtheorem{exes}[thm]{Examples}
\newtheorem{rem}[thm]{Remark}
\newtheorem{que}[thm]{Question}
\newtheorem{reminder}[thm]{Reminder}
\theoremstyle{remark}
\newtheorem{notation}[thm]{Notation}

\newcommand{\Z}{\mathbf{Z}}
\newcommand{\N}{\mathbf{N}}
\newcommand{\R}{\mathbf{R}}
\newcommand{\C}{\mathbf{C}}
\newcommand{\Q}{\mathbf{Q}}
\newcommand{\G}{\mathbf{G}}
\newcommand{\Hom}{\operatorname{Hom}}

\title[Brouwer degree and domination]
{Brouwer degree,
\\
domination of manifolds,
\\
and groups presentable by products}

\author{Pierre de la Harpe}


\date{March 20, 2015, revised April 19, 2016}

\address{Pierre de la Harpe:
Section de math\'ematiques, 
Universit\'e de Gen\`eve,  
C.P.~64, 
CH--1211 Gen\`eve 4. 
e-mail: Pierre.delaHarpe@unige.ch
}

\thanks{I am grateful to Michel Boileau, Michelle Bucher-Karlsson, Jean-Claude Hausmann, 
Dieter Kotschick, Claude Weber, and Susanna Zimmermann, 
for helpful remarks concerning the subject of this exposition,
as well as to
Christoforos Neofytidis and a referee who sent me constructive expert comments 
on a first version.
The paper is an expanded version of an expository talk,
and I am grateful to Ruth Kellerhals for inviting me to give it in the Fribourg Colloquium
on February 24, 2015.}

\begin{document}

\begin{abstract}
For oriented connected closed manifolds of the same dimension,
there is a transitive relation: $M$ \emph{dominates} $N$, or $M \ge N$,
if there exists a continuous map of non-zero degree from $M$ onto $N$.
Section \ref{thecaseofmanifolds} is a reminder on the notion of degree (Brouwer, Hopf),
Section \ref{sectiondomination} shows examples of domination
and a first set of obstructions to domination due to Hopf,
and Section \ref{secsimvol} describes obstructions in terms of
Gromov's simplicial volume.
\par

In Section \ref{sectionproduct} we address the particular question of
 when a given manifold 
can (or cannot) be dominated by a product.
These considerations suggest a notion for groups (fundamental groups),
due to  D.\ Kotschick and C.\ L\"oh: 
a group is \emph{presentable by a product}
if it contains two infinite commuting subgroups
which generate a subgroup of finite index.
The last section shows a small sample of groups which are not presentable by products;
examples include appropriate Coxeter groups.
\end{abstract}

\maketitle

\section{Two definitions of the Brouwer degree}
\label{thecaseofmanifolds}

Degrees as defined below can be traced back to Brouwer \cite{Brou--11}.
Standard references for the Brouwer degree 
include \cite[\S~9]{Pont--59},  \cite{Miln--65}, \cite{Hirs--76}, and \cite[\S~5.5]{Pras--06}.

Consider a fixed dimension $n \ge 1$
and two oriented connected closed smooth manifolds $M$, $N$ of dimension $n$.
Let $f : M \longrightarrow N$ be a smooth map.
Recall that a \emph{regular value} of $f$ is a point $y \in N$
such that the differential of $f$ is invertible at each point of $f^{-1}(y)$.
It is a basic result due to A.\ Sard (1942) that the complement in $N$
of the set of regular values has Lebesgue measure zero;
in particular, the set of regular values of $f$ is dense
\cite[Theorem 3-III]{Brow--35}, and a fortiori non-empty.
For more recent proofs of the results of Brown and Sard, 
see e.g.\ \cite[Page 10]{Rham--60} or \cite[$\S$~2 \& 3]{Miln--65}.
\par

If $y \in N$ is a regular value of $f$, 
then $f$ is a local diffeomorphism around every $x \in f^{-1}(y)$;
it follows that $f^{-1}(y)$ is discrete in $M$;
since $f^{-1}(y)$ is also compact, it is therefore finite.
For $x \in f^{-1}(y)$ define $\varepsilon _x(f)$ to be $1$ if $f$ is orientation preserving at $x$
and  $-1$ if $f$ is orientation reversing at $x$.

\begin{defn}
\label{defdegmanprov}
The \textbf{local degree of $f$ at a regular value $y$} is the integer
\begin{equation}
\tag{$\Sigma$}
\label{eqof1.1}
\deg_{(y)} (f) \, = \, \sum_{x \in f^{-1}(y)} \varepsilon_x (f) .
\end{equation}
\end{defn}

It is a particular case of the definition that
$\deg_{(y)} (f) = 0$ if $f$ is \emph{not} surjective and $y$ not in the image of $M$.

\begin{prop}
\label{indephomot}   
Let $f : M \longrightarrow N$ be a smooth map
and $y$ a regular value of $f$, as above.
\begin{enumerate}[noitemsep,label=(\arabic*)]
\item\label{1DEindephomot}
The local degree $\deg_{(y)} (f)$ defined in (\ref{eqof1.1}) 
is independent on  the choice of the regular value~$y$.
\item\label{2DEindephomot}
For another smooth map $f' : M \longrightarrow N$ homotopic to $f$
and a regular value $y'$ of $f'$, 
the local degrees coincide: $\deg_{(y')} (f') = \deg_{(y)} (f)$.
\end{enumerate}
\end{prop}

\begin{reminder}[fundamental class]
\label{remfc}
For a connected closed manifold $M$ of dimension $n$, 
the homology group $H_n(M; \Z)$ is 
either isomorphic to $\Z$, in which case $M$ is \textbf{orientable}, 
or is $\{0\}$.
When $M$ is orientable, an \textbf{orientation} of $M$
is a choice of generator 
$$
[M] \, \in \, H_n(M; \Z) \simeq \Z ,
$$
called the \textbf{fundamental class} of $M$.
See for example \cite[Chapter 3]{Hatc--02};
more on the fundamental class in \ref{dualFCZ} below.
\end{reminder}

\begin{notation}[induced maps]
\label{notationhomothomol}
Let $X, Y$ be topological spaces, 
$f : X \longrightarrow Y$ a continuous map,
$j$ a non-negative integer, and $A$ an abelian group.
We denote by 
$$
f_j \, : \,  H_j(X; A) \longrightarrow H_j(Y; A)
$$ 
the map induced by $f$ in homology with coefficients in $A$,
and by 
$$
f^j \, : \,  H^j(Y; A) \longrightarrow H^j(X; A)
$$
the map induced by $f$ in cohomology.
\par

If $X, Y$ are connected,
and good enough to have a standard theory of fundamental groups and coverings
(in particular if $X$ and $Y$ are connected manifolds),
and if base points have been chosen,
we denote by 
$$
f_\pi \, : \, \pi_1(X) \longrightarrow \pi_1(Y)
$$
the homomorphism induced on fundamental groups.
\end{notation}

\begin{proof}[Proof of Proposition \ref{indephomot}]
Let $D_y$ be a small embedded open disc (or simplex) in $N$ centred at $y$.
If $D_y$ is small enough, there exists for all $x \in f^{-1}(y)$
an open disc $D_x$ such that the restriction of $f$ to $D_x$
is a diffeomorphism, say $g_x$, from $D_x$ onto $D_y$, and the $D_x$ 's are disjoint.
Observe that $g_x$ is orientation preserving [respectively orientation reversing]
when $\varepsilon_x(f) = 1$ [resp.\ when $\varepsilon_x(f) = -1$]. 
\par

Let $S$ denote an oriented sphere of dimension $n$, with a base point $p$.
Let $\pi_N : N \longrightarrow S$ be an orientation-preserving map obtained by 
mapping $D_y$ diffeomorphically onto the complement of $p$ in $S$,
and collapsing the complement of $D_y$ in $N$ to $p$.
Let $\pi_M : M \longrightarrow S$ be the map 
which coincides with $\pi_N \circ g_x$ on $D_x$, for each $x \in f^{-1}(y)$,
and collapses the complement of $\bigsqcup_{x \in f^{-1}(y)} D_x$ in $M$ to $p$.
The diagram
\begin{equation*}
\dia{
M
\ar@{->>}[d]^{\pi_M}
\ar@{->}[r]^{f}
& N 
\ar@{->>}[d]^{\pi_N}
\\
S
\ar@{->}[r]^{\operatorname{id}}
& S
}
\end{equation*}
commutes (where $\operatorname{id}$ denotes the identity).
\par

The corresponding maps on the top-dimensional homology groups
provide the commutative diagram
\begin{equation*}
\dia{
\Z [M] = H_n(M; \Z) \hskip.2cm
\ar@{->}[d]^{(\pi_M)_n}
\ar@{->}[r]^{f_n}
\hskip.2cm & \hskip.2cm H_n(N; \Z) = \Z [N] 
\ar@{->>}[d]^{(\pi_N)_n}
\\
\Z [S] = H_n(S; \Z) \hskip.2cm
\ar@{->}[r]^{\operatorname{id}}
\hskip.2cm & \hskip.2cm H_n(S; \Z) = \Z [S]
}
.
\end{equation*}
We have $(\pi_N)_n([N]) = [S]$
and $(\pi_M)_n ([M]) =  \sum_{x \in f^{-1}(y)} \varepsilon_x (f) [S] = \deg_{(y)} (f)  [S]$.
Hence
\begin{equation*}
f_n([M]) \, = \, \deg_{(y)} (f) [N] .
\end{equation*}
Since the left-hand side is independent of $y$, Claim (1) holds.
Claim (2) follows because $f'_n = f_n$.
\end{proof}

As a consequence, we can now define the degree of $f$,
denoted by $\deg (f)$, 
to be the number that appears in (\ref{eqof1.1}),
for any regular value $y$ of $f$. 
But, as a consequence of Proposition \ref{indephomot},
more is true, and we can define 
the degree of any \emph{continuous} map between manifolds.

\begin{defn}
\label{defdegman}
Let $M$ and $N$ be two oriented connected closed topological manifolds 
of the same dimension, say $n \ge 1$,
and $f : M \longrightarrow N$ a continuous map.
The \textbf{degree} of $f$ is the integer $\deg (f) \in \Z$ such that
\begin{equation}
\tag{defdeg}
\label{eqof1.5}
f_n([M]) \, = \, \deg (f) [N] \, \in \, H_n(N; \Z) .
\end{equation}
This definition is due to Hopf:
see Formula (3) in \cite{Hopf--30}. 
\end{defn}

In Definition \ref{defdegman}, assume moreover that $M$ and $N$ are smooth.
Claim: $f$ can be approximated by smooth maps.
For a precise definition of the appropriate topology
on the space $\mathcal C (M,N)$ of continuous map from $M$ to $N$
and for a proof of the claim, see \cite{Hirs--76},
Chapter 2,  in particular Theorem 2.6.
Alternatively, see \cite{Pras--06}, Chapter 5, Section~4.4.

Since two smooth approximations $g_0, g_1$ of $f$ 
are homotopic (being close to each other),
their local degrees are equal, by  Proposition \ref{indephomot}.
Hence the local degrees of Definition \ref{defdegmanprov} coincide with
the homologically defined degree of Definition \ref{defdegman}.
\par
As a digression, we note that this strategy applies to spaces of maps from $M$ to $N$
larger than the space of continuous maps; see \cite{BrNi--95}.
\par
Brouwer's definition keeps its merits for many purposes, 
in particular to suggest definitions in a setting 
adapted to functional analysis and differential equations,
as in a very influential 1934 paper by Leray and Schauder \cite{LeSc--34}.

\vskip.2cm

Note also that the connectedness of $M$ in Definition \ref{defdegman}
is not an issue: in case $M$ has several connected components
$M_i$, the degree of $f : M \longrightarrow N$ is the sum of the
degrees of the restrictions $f \vert_{M_i}$.

\vskip.2cm

The two claims of the following proposition are straightforward consequences
of Definition \ref{defdegman}.

\begin{prop}
\label{degcomposition}
Let $M, N, P$ be oriented connected closed manifolds of the same dimension $n$.

\vskip.2cm

(1)
Let $f : M \longrightarrow N$ and $g : N \longrightarrow P$ be continuous maps.
Then
$
\deg (g \circ f) =
\deg(f)  \deg (g) .
$

\vskip.2cm

(2)
Suppose that $M = M_1 \times M_2$ and $N = N_1 \times N_2$ are product manifolds,
where $M_1, M_2, N_1, N_2$ are oriented connected closed manifolds
with $\dim (M_1) = \dim (M_2)$ and $\dim (N_1) = \dim (N_2)$,
and that $f : M \longrightarrow N$ is a product of two continuous maps $f_j : M_j \longrightarrow N_j$ ($j=1,2$).
Then $\deg (f) = \deg (f_1) \deg (f_2)$.
\end{prop}

\begin{rem}
\label{3DEindephomot}
Besides \ref{1DEindephomot} and \ref{2DEindephomot}
in Proposition \ref{indephomot},
the following important step 
in the standard proof of the proposition
has an independent interest (see \cite[\S~5, Lemma 1]{Miln--65}):
\begin{enumerate}[noitemsep]
\item[]
\emph{suppose that $M$ is the boundary of 
an oriented connected smooth manifold with boundary, say $X$, of dimension $n+1$, 
and that $f$ extends to a smooth map $X \longrightarrow N$;
then $\deg (f) = 0$.}
\end{enumerate}
Indeed, $f$ extends if and only if $\deg (f) = 0$
\cite[Chap.\ 5, Th.\ 1.8]{Hirs--76}.
\end{rem}

\begin{rem}[history, homology and groups]
The formulation above is completely standard nowadays,
and rather far from the original formulation of Brouwer,
for several reasons. 
\par

One is that Brouwer used ``piecewise linear'' approximations (rather than smooth approximations)
of continuous mappings.
\par
Another one is that the notion of group was not used in homology theory
before the mid 20's, when homology invariants were promoted 
from numbers (Riemann, Betti, Poincar\'e) to groups,
under the strong influence of Emmy Noether 
and a few others (Hopf, Mayer, Vietoris).
E.\ Noether gave lectures on the subject in G\"ottingen, 
but apparently did not publish anything but
a 14-lines report\footnote{Noether's report
is about the structure of finitely generated abelian groups.
In its 14 lines, topology is mentioned in the last phrase only:
``Der Gruppensatz erweist sich so als der einfachere Satz;
in den Anwendungen des Gruppensatzes 
-- z.B.\ Bettische und Torsionszahlen in der Topolgie --
ist somit ein Zur\"uckgehen auf die Elementarteilertheorie nicht erforderlich.''
}
dated 27.\ Januar 1925 \cite{Noet--25}; see \cite{Hirz--96}.
\par

Heinz Hopf had received his doctorate in 1925 in Berlin, from Erhard Schmidt.
He went then to G\"ottingen as a young ``postdoc'',
and met E.\ Noether and P.\ Alexandroff.
The book ``Topologie'' by Alexandroff and Hopf
was published in 1935 \cite{AlHo--35}, with the dedication
\begin{center}
L.E.J.\ Brouwer gewidmet.
\end{center}
\par
Topology was not quite a respectable subject in these times;
note for example that it plays a rather minor role in Hilbert's problems.
It seems that Hilbert was not comfortable
with Poincar\'e's highly intuitive approach to topology, 
involving ``mistakes'', corrections, and imprecise statements.
From the middle 20's onwards, topology became somehow respectable \cite{Eckm--06};
this owes a lot to the work of Hopf, and a few others including Alexander and de Rham.
\par

Thom gave his view on the Brouwer degree in modern differential topology \cite{Thom--71}.
\end{rem} 

\begin{rem}[dual fundamental classes]
\label{dualFCZ}
Let $M, N$ are orientable connected closed manifolds
of the same dimension, $n \ge 1$, 
and $f : M \longrightarrow N$ a continuous map.

\vskip.2cm

It follows from the Universal Coefficient Theorem for cohomology
that we have an isomorphism $H^n(M; \Z) \simeq \Hom (H_n(M; \Z), \Z)$,
and therefore that $\Hom (H_n(M; \Z), \Z)$ is isomorphic to $\Z$.
The \textbf{dual fundamental class} is the element $[M]^* \in H^n(M; \Z)$
which maps $[M] \in H_n(M; \Z)$ to $1 \in \Z$.
Similarly for $[N]^* \in H^n(M; \Z)$.
\par

The induced map on $H^n(\cdot; \Z)$ is the transposed
$$
f^n \, : \, 
\operatorname{Hom}(H_n(N; \Z), \Z) \, \longrightarrow \, \operatorname{Hom}(H_n(M; \Z), \Z),
\hskip.2cm 
c \, \longmapsto \, c \circ f_n
$$
of the map $f_n : H_n(M; \Z) \longrightarrow H_n(N; \Z)$ in homology. 
Hence we have a dual formulation
\begin{equation}
\tag{defdegco}
\label{eqof1.9}
(f)^n([N]^*) \, = \, \deg (f) [M]^* \, \in \, H^n(M; \Z) 
\end{equation}
of the definition via (\ref{eqof1.5}) in \ref{defdegman}.

\vskip.2cm

A particular case of the Universal Coefficient Theorem for homology establishes that
the natural map
$H_n(M; \Z) \otimes \Q \longrightarrow H_n(M; \Q)$  
is an isomorphism.
The image of $[M] \otimes 1$ is not zero,
and is consequently a basis of the one-dimensional $\Q$-vector space $H_n(M; \Q)$;
it is denoted by $[M]$ again.
Similarly for $[N] \in H_n(N; \Q)$.
\par

For a map $g$ of  $M$ to itself, the degree appears
as the multiplication factor of the induced endomorphism of $H_n(M; \Q)$, 
without reference to $H_n(M; \Z)$;
but this conceals that $\deg (g)$ is always an \emph{integer}.
For $f : M \longrightarrow N$,
the degree of $f$ is again given by
$f_n([M]) = \deg (f) [N] \in H_n(N; \Q)$;
but there is in this case a reference to $H_n(\cdot; \Z)$, 
because $H_n(M; \Q)$ and $H_n(N; \Q)$
\emph{have to} be viewed together with distinguished elements,
which are images of  fundamental classes defined in integral homology.
\end{rem}

\begin{exes}
\label{exdeg}
Let $M,N$ be as in Definition \ref{defdegman}.

\vskip.2cm

\label{1DEexdeg}
(1)
For every $d \in \Z$, there  exists a continuous map $M \longrightarrow S^n$
of degree $d$.  
\par

This holds if $M$ is the circle $S^1 = \{ z \in \C \mid \vert z \vert = 1 \}$,
because the map $z \longmapsto z^d$ of the circle into itself is of degree $d$.
This holds if $M$ is a sphere of larger dimension $n$, by induction:
indeed, if $\varphi$ is a continuous map of a $n$-sphere into itself,
its suspension is a continuous map 
of a $(n+1)$-sphere into itself of the same degree as $\varphi$.
Consider now the general case.
There exists a map from $M$ to $S^n$ of degree $1$
(see the proof of Proposition \ref{indephomot});
the composition with a map from $S^n$ to itself of degree $d$
is a map from $M$ to $S^n$ of degree $d$.  
\par
It is a famous theorem of Hopf \cite{Hopf--27} that:
\begin{center}
\emph{two continuous maps from $M$ to the $n$-sphere are homotopic
\\
if and only if their degrees coincide.}
\end{center}
Indeed, Whitney proved later (1937) the following generalization:
if $X$ is any connected polyhedron of dimension at most $n$, 
homotopy classes of continuous maps $f : X \longrightarrow S^n$
are in bijection with $H^n(X; \Z)$, the homotopy class of $f$
corresponding to the image $f^n([S^n]^*) \in H^n(X; \Z)$,
where $[S^n]^*$ is the dual fundamental class of the $n$-sphere.

\vskip.2cm

(2)
\label{2DEexdeg}
If there exists a homeomorphism $f$ from $M$ onto $N$,
then $\deg(f) = 1$ if $f$ is orientation-preserving,
and $\deg(f) = -1$ if $f$ is orientation-reversing.
The conclusion carries over to the case of a homotopy equivalence
$f : M \longrightarrow N$.
\par

There are known examples of pairs $(M,N)$ of non-homeomorphic manifolds
for which there exists a homotopy equivalence from $M$ to $N$.
In dimension $3$, there is a standard example,
that of the pair of lens spaces $L(7,1)$, $L(7,2)$.

\vskip.2cm

\footnotesize
\emph{Reminder on lens spaces.}
For $p \ge 2$ and $k \in \{1, \hdots, p\}$ coprime with $p$,
let $L(p, k)$ denote
the quotient of the $3$-sphere $\{ (z_1, z_2) \in \C^2 \mid \vert z_1 \vert^2 + \vert z_2 \vert^2 = 1\}$
by an action of the cyclic group $\Z / p\Z$ for which a generator acts by
$(z_1, z_2) \longmapsto (\exp (2 \pi i / p) z_1,  \exp(2 \pi i k/p) z_2)$.
We have $\pi_1(L_p(p,k)) \simeq \Z / p\Z$.
\par
In the mid 30's,
Reidemeister has found a necessary condition for $L(p,k)$ and $L(p,\ell)$ to be PL-homeomorphic;
in particular $L(7,1)$ and $L(7,2)$ are NOT homeomorphic.
Rueff, and slightly later Franz,
have studied conditions for $L(p,k)$ and $L(p,\ell)$ to have the same homotopy type;
in particular $L(7,1)$ and $L(7,2)$ DO have the same homotopy type.
More on lens spaces in \cite{Rham--67, Miln--66, Volk--13}.
\par
The following result uses a lot, including Perelman's work:
Let $M$, $N$ be two irreducible connected closed $3$-manifold 
which have the same homotopy type.
At least one of the following is true:
(i) $M$ and $N$ are homeomorphic,
(ii) $M$ and $N$ are lens spaces.
[From a conversation with Claude Weber.]
\normalsize

\vskip.2cm

\label{3DEexdeg}
(3)
If there exists a covering $f : M \longrightarrow N$ with $k$ leaves,
then $\deg (f) = k$ or $\deg (f) = -k$.

This carries over to branched coverings,
i.e.\ to maps $f : M \longrightarrow N$ for which there is a subset $B \subset M$
``small enough'' (e.g. of codimension at least two)
such that the restriction of $f$ to $M \smallsetminus B$ is a covering
over $N \smallsetminus f(B)$.

\vskip.2cm

\label{4DEexdeg}
(4)
Suppose that $M$ is a connected sum $M = N \sharp P$ 
of two oriented connected closed manifolds $N, P$ of the same dimension.
The manifold obtained from $M$ by collapsing $P$ to a point
can be identified with $N$, so that there is a map $f : M \longrightarrow N$
well defined up to homotopy.
Then $\deg (f) = 1$.
\par
In particular, since $M$ and $M \sharp S^n$ are homeomorphic,
we obtain again a map $M \longrightarrow S^n$ of degree $1$, as in (1).

\vskip.2cm

\label{5DEexdeg}
(5)
Suppose that $\Sigma_g$ and $\Sigma_h$ are 
oriented connected closed surfaces, of genus $g$ and $h$ respectively.
The following properties are equivalent: 
\begin{enumerate}[noitemsep]
\item[$\cdot$]
$g \ge h$, 
\item[$\cdot$]
there exists a continuous map $f$ from $\Sigma_g$ to $\Sigma_h$ of non-zero degree,
\item[$\cdot$]
there exists a continuous map $f$ from $\Sigma_g$ to $\Sigma_h$ of degree $1$.
\end{enumerate}

\vskip.2cm

The equalities in (2), (3) and (4) follow from the definitions,
and the equivalences of (5) are left as an exercise for the reader.
[Hint: to exclude the existence of maps with non-zero degrees,
use Proposition \ref{propdegmanif}\ref{2DEpropdegmanif} below;
see Example \ref{exapresUmkehr}(3).]

\vskip.2cm

\label{6DEexdeg}
(6)
For every continuous map $f$ of the complex projective space $P^n_\C$ to itself,
the degree is an integer of the form $d^n$, for some $d \in \Z$.
\par
Indeed, on the one hand, 
if $h$ denotes the standard generator of $H^2(P^n_\C; \Z) \simeq \Z$, 
so that $h^n$ is the dual fundamental class $[P^n_\C]^* \in H^{2n}(P^n_\C; \Z)$,
and if $d$ is the integer such that $f^2(h) = dh$,
then $f^{2n}([P^n_\C]^*) = \left( f^2(h) \right)^n = (dh)^n = d^n  [P^n_\C]^*$.
\par
On the other hand, the map given in homogeneous coordinates by
$$
P^n_\C \longrightarrow P^n_\C , \hskip.2cm
(z_0:z_1: \cdots :z_n) \longmapsto (z_0^d : z_1^d : \cdots : z_n^d) 
$$
has degree $d^n$ for every $d \in \N$.
\end{exes}

There are many important statements 
with simple proofs involving the notion of Brouwer degree.
Here is a standard sample.

\begin{thmBr}
\label{thmBrouwer}
Every continuous mapping $f$ from the $n$-disc $D^n$ into itself has a fixed point.
\end{thmBr}

\begin{proof}
We identify $D^n$ to the upper hemisphere, defined by $x_n \ge 0$,  of
$$
S^n \, = \,  \Big\{ (x_0, \hdots, x_n) \in \R^n
\hskip.2cm \Big\vert \hskip.2cm  
\sum_{i=0}^n x_i^2 = 1 \Big\} .
$$
Define a transformation $g$ of $S^n$ by 
$$
g(x) \, = \, \left\{
\aligned 
f(x) \hskip3.3cm &\text{if} \hskip.2cm x \in D^n \subset S^n
\\
f(x_0, \hdots, x_{n-1}, -x_n) \hskip.5cm &\text{if} \hskip.2cm
x = (x_0, \hdots, x_n) \notin D^n, \hskip.2cm x \in S^n .
\endaligned
\right.
$$
Then $\deg (g) = 0$, because $g$ is not surjective. 

Every fixed-point free transformation $h$ of $S^n$ 
has degree $(-1)^{n+1}$,
because $h$ is then homotopic to the antipodal map $a : x \longmapsto -x$,
of degree $(-1)^{n+1}$; indeed,
$H(t,x) = \frac{ th(x) - (1-t)x }{ \Vert th(x) - (1-t)x \Vert }$
defines a homotopy from $H(1, \cdot) = h$ to $H(0, \cdot) = a$,
so that $\deg (h) = \deg (a) = (-1)^{n+1} \ne 0$.

Hence $g$ has a fixed point, and every fixed point of $g$ is also a fixed point of $f$.
\end{proof}

\begin{thmBIDT}
\label{thmInvarianceDomain}
Let $U$ be an open subset of $\R^n$ and $\varphi : U \longrightarrow \R^n$
a continuous injective map; then $f(U)$ is open.
\end{thmBIDT}

For a proof using the notion of Brouwer degree, see e.g.\ 
\cite[$\S$~6.2]{Tao--14}.

\begin{rem}[other manifolds]
In Definition \ref{defdegman}, the requirements about $M, N$ being closed can be relaxed.

\vskip.2cm

(1)
Consider manifolds with boundaries.
Let $M$ be a connected compact $n$-manifold, possibly with non-empty boundary $\partial M$.
What has been recalled in \ref{remfc} extends as follows:
the homology group $H_n(M, \partial M; \Z)$ is
either isomorphic to $\Z$, in which case $M$ is \textbf{orientable}, 
or is $\{0\}$.
When $M$ is orientable, an \textbf{orientation} of $M$
is a choice of generator $[M, \partial M]$ in $H_n(M, \partial M; \Z) \simeq \Z$,
and $[M, \partial M]$ is then called the \textbf{fundamental class} of $M$.
\par

Let $M,N$ be two such oriented manifolds.
The degree of a continuous map $f$ from $M$ to $N$ such that $f(\partial M) \subset \partial N$
is the integer $d$ such that
\begin{equation}
\tag{$\sharp$}
\label{eqof1.15}
f_n([M, \partial M]) \, = \, d [N, \partial N] ,
\end{equation}
where $f_n :  H_n(M, \partial M; \Z) \longrightarrow H_n(N, \partial N; \Z)$
is induced by $f$.
\par

Supposer moreover that $M$ and $N$ are smooth.
A map $f : (M, \partial M) \longrightarrow (N, \partial N)$ as above 
is homotopic to a smooth map, say $g$.
Such a smooth map has regular values in $N \smallsetminus \partial N$,
and the degree $\deg(f) = \deg(g)$ can be computed as in Equation (\ref{eqof1.1})
of Definition \ref{defdegmanprov}.

\vskip.2cm

(2)
Let now $M, N$ be two orientable connected $n$-manifolds without boundary
(note that $M, N$ need not be compact).
If $M, N$ are smooth,
Definition  \ref{defdegmanprov}  and Proposition \ref{indephomot}
carry over to \emph{proper} smooth maps from $M$ to $N$.
\par

If $M$ is not compact, recall that $H_n(M; \Z) = \{0\}$.
But Definition \ref{defdegman} can be adapted
by using an appropriate notion of homology
for which an orientable manifold need not be compact to have a fundamental class,
for example by using the locally finite homology exposed in \cite[Part III]{Geog--08}.
\end{rem}

\section{Domination}
\label{sectiondomination}

Let $M, N$ be oriented connected closed manifolds of dimension $n$.

\begin{defn}
\label{defdomin}
The manifold $M$ \textbf{dominates} $N$ 
if there exists a continuous map of non-zero degree from $M$ to $N$.
More precisely, for $d \ge 1$, the manifold $M$ \mbox{$d$-dominates} $N$
if there exists a continous map of degree $\pm d$ from $M$ to $N$.
\end{defn}

\footnotesize
\noindent \emph{Watch out!}
Let $X,Y$ be topological spaces.
Then $X$ \emph{dominates} $Y$ in the sense of J.H.C.\ Whitehead
if there exist continuous maps $i : Y \longrightarrow X$ and $r : X \longrightarrow Y$
such that  $r \circ i$ is homotopic to $\operatorname{id}_X$.
If $X$ dominates $Y$ in this sense, 
then $r_j : H_j(X; \Z) \longrightarrow H_j(Y; \Z)$ is surjective
for all $j \ge 0$.
See \cite{Whad--48} and \cite[Appendix]{Hatc--02}.
\par
Below, domination will always refer to Definition \ref{defdomin},  and not to Whitehead's notion.
\normalsize

\vskip.2cm

The following structuring question has been asked by Gromov in a talk from 1978
(according to \cite{CaTo--89}).
It is of course present, implicitly but strongly, in \cite{Hopf--30},
and also later in \cite{MiTh--77} and \cite{Grom--82}.

\begin{que}
\label{questionGrom}
Let $M,N$ be as above. 
Can one decide whether or not $M$ dominates $N$ ?
\end{que}

\begin{exes}
\label{exdomin}
Each of (1) to (5) below is related to the same number in Example \ref{exdeg} above.

\vskip.2cm
 
(1) 
Every $n$-manifold $M$ as above dominates the $n$-sphere.
Compare with the proof of Proposition \ref{indephomot},
and with (4) in Example \ref{exdeg}.

\vskip.2cm

(2) 
Two manifolds of the same homotopy type dominate each other.

\vskip.2cm

(3) 
A covering or a branched covering with a finite number of sheets is a domination.
\par

For a specific example, consider a non-trivial finite subgroup $\Gamma$
of the $3$-sphere $S^3$,
identified to the special unitary group $\operatorname{SU}(2)$:
set $N = S^3/\Gamma$ and $d = \vert \Gamma \vert$.
There exists a covering $S^3 \longrightarrow N$ of degree $d$
(or $-d$ for ``the other'' orientation of $N$)
and a smooth map $N \longrightarrow S^3$ of degree $1$ (as in (1) above).
Hence each of $N, S^3$ dominates the other. 

\vskip.2cm

(4) 
A connected sum of two oriented manifolds dominates each of them.

\vskip.2cm

(5)  
In particular,
an oriented connected closed surface dominates every such surface of lower genus.

\vskip.2cm

(6)
For every oriented connected closed $3$-manifold $N$, 
there exists a \emph{hyperbolic} $3$-manifold $M$ which dominates $N$.
This is a consequence of results of Sakuma and Brooks \cite{Broo--85},
revisited in \cite{Mont--87}.
\par
Theorem \ref{thmGaifullin} for $n=3$ gives a stronger result.

\vskip.2cm

(7)
For $r \ge 1$, let $P^r_\C$ denote the complex projective space
of complex dimension $r$.
Consider two integers $p,q \ge 1$ and their sum $n = p+q$.
The product $P^p_\C \times P^q_\C$ dominates $P^n_\C$. 
\par
Here is one way to construct a holomorphic map
$P^p_\C \times P^q_\C \longrightarrow P^n_\C$ of Brouwer degree $\binom{n}{p}$.
The starting ingredient is the Segre embedding 
$$
\sigma_{[(p+1)(q+1)-1]} \, : \, \left\{
\aligned
 P^p_\C \times P^q_\C \hskip2cm   \, &\longrightarrow \,  \hskip.3cm P^{(p+1)(q+1)-1}_\C
 \\
\big( (u_0 : u_1 : \cdots : u_p) , (v_0 : v_1 : \cdots : v_q ) \big) \, &\longrightarrow \, 
( \cdots : u_iv_j : \cdots )
\endaligned \right.
$$
(the notation $( \cdots : \cdots : \cdots )$ indicates homogeneous coordinates).

For $k$ with $(p+1)(q+1)-1 > k \ge n$, define inductively
a mapping $\sigma_{[k]}$ from the image of $\sigma_{[k+1]}$ to $P^k_\C$ as follows:
choose a point $a \in P^{k+1}_\C$ not in the image of $\sigma_{[k+1]}$,
and define $\sigma_{[k]}(x)$ as the intersection with $P^k_\C$ 
(identified to a hyper surface of $P^{k+1}_\C$)
of the line containing $a$ and $x$, for all $x$ in the image of $\sigma_{[k+1]}$.
It can be checked that the composition 
$$
\sigma \,  := \,  \sigma_{[n]} \circ \sigma_{[n+1]} \circ \cdots \circ 
\sigma_{[(p+1)(q+1)-2]} \circ \sigma_{[(p+1)(q+1)-1]}
\, : \, 
P^p_\C \times P^q_\C \longrightarrow P^n_\C
$$
is a map of degree $\binom{n}{p}$.

\vskip.2cm

Let us consider the special case
$P^1_\C \times P^2_\C \longrightarrow P^3_\C$.
There are rational maps
$$
\aligned
\sigma_{[5]} \, : \, 
& P^1_\C \times P^2_\C \, \longrightarrow \, P^5_\C , \hskip.5cm
\big( (u:v) , (x:y:z) \big)
\hskip.1cm  \longmapsto \, (ux:uy:uz:vx:vy:vz) ,
\\
\rho^5_4 \, : \, 
& P^5_\C \, \DashedArrow \,  P^4_\C , \hskip.5cm 
(x_0 : x_1 : x_2 : x_3 : x_4 : x_5) 
\hskip.1cm  \vert  - \to \, 
(x_1 : x_2 : x_3 : x_4 : x_5-x_0) ,
\\
\rho^4_3 \, : \, 
& P^4_\C \, \DashedArrow \,  P^3_\C , \hskip.5cm 
(y_0 : y_1 : y_2 : y_3 : y_4) 
\hskip.1cm  \vert  - \to \, 
(y_1 : y_2-y_o : y_3 : y_4) .
\endaligned
$$
Observe that 
\begin{enumerate}[noitemsep]
\item[$\cdot$]
$\rho^5_4$ is not defined in $(1:0:0:0:0:1)$, 
which is a point outside $\operatorname{Image}(\sigma_{[5]})$, 
\item[$\cdot$]
$\rho^5_4 \circ \sigma_{[5]} :  \big( (u:v) , (x:y:z) \big) \longmapsto (uy : uz : vx : vy : vz-ux )$
is defined on the whole of $P^1_\C \times P^2_\C$, 
\item[$\cdot$]
$\rho^4_3$ is not defined in $(1:0:1:0:0)$,
which is a point outside $\operatorname{Image}(\rho^5_4 \circ \sigma_{[5]})$, 
\end{enumerate}
so that
$$
\sigma := 
\rho^4_3 \circ \rho^5_4 \circ \sigma_{[5]} \, : \, 
\left\{ \aligned
P^1_\C \times P^2_\C \hskip1cm &\longrightarrow \hskip2cm P^3_\C
\\
\big( (u:v) , (x:y:z) \big) \, &\longmapsto \, (uz : vx-uy : vy: vz-ux)
\endaligned \right.
$$
is indeed defined everywhere.
It can be checked that $\sigma$ is of degree $3$, using
Definition (\ref{eqof1.1}) of \ref{defdegmanprov}
and the regular value $(1 : 1 : 1 : 1) \in P^3_\C$.
\par

I am grateful  to J.C.\ Hausmann and S.\ Zimmermann for discussions concerning this example.

\end{exes}

Propositions \ref{propdegmanif} and \ref{propsimplicialvolmap}
will establish \emph{obstructions} to the domination of one manifold by another.
\par

Before Proposition \ref{propdegmanif},
let us recall one definition of Hopf's Umkehrungshomomorphismus \cite{Hopf--30}.

\begin{defn}
\label{defUmkehr}
Let $A$ be an abelian group.
Let $M, N$ be oriented connected closed manifolds of the same dimension. 
For $j \in \{1, \hdots, n\}$, let
$$
P_j \, : \, H^{n-j}(M; A) \longrightarrow H_j(M; A) , 
\hskip.2cm  c \longmapsto [M] \cap c
$$
denote the Poincar\'e duality isomorphism, where $\cap$ stands for the cap product.
Let $f : M \longrightarrow N$ be a continuous map.

The \textbf{Umkehrungshomomorphismus} $f_!$ is defined 
as the composition that makes the diagram
\begin{equation}
\tag{Umkehr}
\label{eqof2.1}
\dia{
H^{n-j}(M; A) \hskip.2cm
\ar@{<-}[d]^{f^{n-j}}
\ar@{->}[r]^{P_j(M)}_{\approx}
& \hskip.2cm H_j(M; A) 
\ar@{<.}[d]^{f_!}
\\
H^{n-j}(N; A) \hskip.4cm
\ar@{<-}[r]^(0.58){P_j(N)^{-1}}_{\approx}
& \hskip.2cm H_j(N; A)
}
\end{equation}
commute.
Note that $f_!$ is a ``wrong way map'' in homology;
note also that neither $f_*$ nor $f_!$ need respect
the intersection product in homology.
\end{defn}

\begin{rem}
(i)
The previous definition is highly anachronic, using cohomology, 
as Freudenthal and others did later \cite{Freu--37}.
Indeed, cohomology with its product appeared in 1935 only, 
in communications by Alexander and Kolmogorov 
at the ``Premi\`ere Conf\'erence internationale de Topologie'',
in Moscow,  from 4th to 10th of September,  1935
\cite{Mosc--35, Whey--88}.
\par
In 1930, Hopf could only use homology.
We refer to \cite{Hilt--88} for historical comments,
and to \cite[Section 3.4.1]{Geig--08}
for an exposition of the Umkehrungshomomorphismus in homology.

\vskip.2cm

(ii)
Despite (i), 
we like to think that the following proposition is due to  Hopf.
For (1) and (2) below, 
see respectively (B) on Page 86 and Satz Ib on Page 77 in \cite{Hopf--30}.

\vskip.2cm

(iii)
Digression:
similarly, it was before the discovery of cohomology, and more precisely in 1931,
that de Rham established a conjecture of E.~Cartan and showed that
the dimension of the ``de Rham cohomology'' $H^j_{\text{dR}}(M; \R)$
of a closed $n$-manifold $M$ coincides with its $j$th Betti number.
\end{rem}

\begin{prop}[Hopf]
\label{propdegmanif}
Let $M,N$ be oriented connected closed manifolds of the same dimension, $n$.
Assume that there exists a continuous map $f : M \longrightarrow N$
of degree $d := \deg (f) \ne 0$, i.e.\ that $M$ dominates $N$.
\begin{enumerate}[noitemsep,label=(\arabic*)]
\item\label{1DEpropdegmanif}
The image of the homomorphism 
$$
f_\pi \, : \, \pi_1(M) \longrightarrow \pi_1(N)
$$
is a subgroup of $\pi_1(N)$ of finite index, say $c$,
and $c$ divides $\vert d \vert$.
\item\label{2DEpropdegmanif}
For each $j \in \{0, 1, \hdots, n\}$, the map 
$$
f_j \, : \, H_j(M; \Q) \longrightarrow H_j(N; \Q)
$$
in homology is surjective.
In particular, there is an inequality
$$
\dim (H_j(M; \Q)) \, \ge \,  \dim (H_j(N; \Q))
$$
for the Betti numbers of $M$ and $N$.
\item\label{3DEpropdegmanif}
For each $j \in \{0, 1, \hdots, n\}$, the map 
$$
f^j \, : \, H^j(N; \Q) \longrightarrow H^j(M; \Q)
$$
in cohomology is injective.
\item\label{4DEpropdegmanif}
More precisely, the composition $f_* \circ f_! : H_*(N; \Q) \longrightarrow H_*(N; \Q)$
is multiplication by the number $d$.
\end{enumerate}
\end{prop}

\noindent \emph{Note.}
We could equally use $\R$ instead of $\Q$ as coefficients.
\par

Concerning \ref{3DEpropdegmanif},
we will see below examples where it is important that
$f^*$ embeds the cohomology \emph{algebra} $H^*(N; \Q)$
as a subalgebra of $H^*(M; \Q)$.
In contrast,
the induced maps of \ref{2DEpropdegmanif} \emph{need  not}
preserve the intersection product (think of a map $S^1 \times S^1 \longrightarrow S^2$
of degree $1$).

\begin{proof}
\ref{1DEpropdegmanif}
Let $\Delta$ be the image of the induced homomorphism 
$f_\pi$ from $ \pi_1(M)$ to $\pi_1(N)$,
and let $p : \widetilde N \longrightarrow N$ the covering of $N$ corresponding to $\Delta$.
The map $f$ factorizes as the composition of a lift $\widetilde f : M \longrightarrow \widetilde N$ of $f$
with $p$.
\par

If $\Delta$ were of infinite index in $\pi_1(N)$, 
the manifold $\widetilde N$ would be non-compact,
the homology group $H_n(\widetilde N; \Z)$ the zero group, 
and $f_n = p_n \circ \widetilde f_n$ the zero map,
contradicting $d \ne 0$.
Hence $\Delta$ is of finite index, say of index $c$, in $\pi_1(N)$.
Since the covering $p$ is of degree $c$, 
the degree $d$ of the composition $f = p \circ \widetilde f$
is the product of $\deg (\widetilde f)$ and $c$.

\vskip.2cm

\ref{2DEpropdegmanif}-\ref{3DEpropdegmanif}-\ref{4DEpropdegmanif}
For $j \in \{0, 1, \hdots, n\}$, consider the map $F_j$ defined as that for which the diagram
\begin{equation}
\tag{F}
\label{eqof2.2}
\dia{
H^{n-j}(M; \Q) \hskip.2cm
\ar@{<-}[d]^{f^{n-j}}
\ar@{->}[r]^{P_j(M)}_{\approx}
& \hskip.2cm H_j(M; \Q) 
\ar@{->}[d]^{f_j}
\\
H^{n-j}(N; \Q) \hskip.4cm
\ar@{.>}[r]^(0.58){F_j}
& \hskip.2cm H_j(N; \Q)
}
\end{equation}
commutes.
By naturality, $F_j$ is the cup product
$$
c \longmapsto f_n([M]) \cap c .
$$
Since $f_n([M]) =  d [N]$ with $d \ne 0$, 
the map $F_j$ is a non-zero multiple of the Poincar\'e duality isomorphism $P_j(N)$,
and in particular $F_j$ is an isomorphism.
It follows that $f^{n-j}$ is injective and $f_j$ surjective.
\par
By comparing the diagrams (\ref{eqof2.1}) and (\ref{eqof2.2}),
we see that $f_* \circ f_!$ is multiplication by $d$.
\end{proof}

\begin{rem}
Domination is often \emph{thought of} as a \textbf{partial order} on
the set $\mathcal M(n)$ of (homotopy types of) 
oriented connected closed manifolds of some dimension $n$,
for which $S^n$ is an absolutely minimal element.
But it is not quite a partial order,
as shown by $S^3$ and $S^3 / \Gamma$ as in Example
\ref{exdomin}(3).
\par

Note that, in general, two manifolds in $\mathcal M(n)$
are not comparable for this ``partial order'', because neither of them dominates the other. 
\par

In the case of aspherical $3$-manifolds,
this order has been extensively studied by Wang \cite{Wang--91};
see also \cite{KoNe--13}.
For non-hyperbolic $4$-manifolds, see \cite{Neof}.

This ``partial order'' could be extended to pairs of oriented connected closed manifolds
of possibly different dimensions, as suggested in \cite{CaTo--89}: 
$M$ \textbf{dominates} $N$ if there exists a continuous map
from $M$ to $N$ which is surjective in rational homology.
\end{rem}

\begin{exes}
\label{exapresUmkehr}
(1) 
If $N$ is a manifold  dominated by $S^n$,
then $\pi_1(N)$ is finite, and $N$ is a rational homology sphere,
respectively by \ref{1DEpropdegmanif} and \ref{3DEpropdegmanif}
in Proposition \ref{propdegmanif}.
\par
Observe that $S^3$ dominates infinitely many pairwise non-homeomorphic manifolds,
quotients of $S^3$ by finite cyclic groups of all orders.

\vskip.2cm

(2)
The converse to a particular case of (1) holds as follows:
\emph{A simply connected rational homology $n$-sphere $N$ is dominated by $S^n$.}
This was shown to me by C.\ Neofytidis, with the following argument.
For $j \in \{2, \hdots, n-1\}$, the homology groups $H_j(N; \Z)$
are in the Serre class $\mathcal C$ of finite abelian groups
(see the discussion and the references in \cite{Neof--14}).
The Hurewicz theorem modulo the class $\mathcal C$ implies that
the Hurewicz homomorphism $h : \pi_n(N) \longrightarrow H_n(N; \Z)$
is an isomorphism modulo $\mathcal C$;
in particular, the image of $h$ is not $\{0\}$, and thus $S^n$ dominates $N$.

\vskip.2cm

(3)
Let $\Sigma_g$ and $\Sigma_h$ be 
oriented connected closed surfaces, of genus $g$ and $h$ respectively,
as in Example \ref{exdeg}(5).
If there exists a continuous map $\Sigma_g \longrightarrow \Sigma_h$
of non-zero degree, then $g \ge h$, 
by Proposition \ref{propdegmanif}\ref{2DEpropdegmanif}.
This together with Example \ref{exdomin}(5) shows that
$\Sigma_g$ dominates $\Sigma_h$ if and only if $g \ge h$.

\vskip.2cm

(4) 
None of the manifolds $S^2 \times S^4$ and $P^3_{\C}$ dominates the other.
This follows from \ref{3DEpropdegmanif} in Proposition \ref{propdegmanif},
but not from \ref{2DEpropdegmanif}.
\par
Indeed, the cohomologies
$$
\aligned
H^*(S^2 \times S^4; \Q) \, &= \,  1 \Q \oplus c^S_2 \Q  \oplus c^S_4 \Q  \oplus [S^2 \times S^4]^* \Q
\\
&\simeq \, \Q[x,y] / (x^2, y^2) ,
\\
H^*(P^3_{\C}; \Q) \, &= \,  1  \Q \oplus c^P_2 \Q  \oplus c^P_4 \Q  \oplus [P^3_{\C}]^* \Q 
\\
&\simeq \, \Q[z] / (z^4) ,
\endaligned
$$
are both generated, additively, by classes of degree $0, 2, 4, 6$.
But neither of these algebras injects as a subalgebra in the other, 
since $(c^S_2)^2 = 0$ and $(c^P_2)^2 = c^P_4 \ne 0$.

\vskip.2cm

(5)
Recall first that, if $M = M_1 \sharp M_2$ 
is a connected sum of two oriented connected manifolds,
$H_j(M; \Z) \simeq H_j(M_1; \Z) \oplus H_j(M_2; \Z)$ for all $j$ with $1 \le j \le n-2$~;
if $M_1$ and $M_2$ are orientable, this holds also for $j = n-1$.
The proof involves 
the long exact sequence for homology of the pair $(M_1 \sharp M_2, S^{n-1})$, 
showing that $H_j(M_1 \sharp M_2; \Z)$ and $H_j(M_1 \vee M_2; \Z)$ are isomorphic for $j \le n-2$,
and the Mayer-Vietoris exact sequence showing that
$H_j (M_1 \vee M_2; \Z) \simeq H_j(M_1; \Z) \oplus H_j(M_2; \Z)$ for all $j \ge 1$.
Recall that the wedge sum 
$M_1 \vee M_2$ is the space obtained from the disjoint union $M_1 \sqcup M_2$
by identifying \emph{one} point of $M_1$ with one of $M_2$.
\par

Let $M$ be the connected sum of two copies of $P^2_{\C}$ with the same orientation, 
so that the intersection form is definite.
Let $N$ be the connected sum of two copies of $P^2_{\C}$ with opposite orientations, 
so that the intersection form is indefinite.
Then $M$ and $N$ have the same fundamental group (which is trivial)
and the same additive homology and cohomology groups.
\par

The manifold $M$ does not dominate $N$, 
and $N$ does not dominate $M$, 
because none of
$$
\aligned
H^*(M, \Q) \, = \, 
H^*(P^2_\C \sharp P^2_\C; \Q) \, &\simeq \, \Q [x,y]/(x^2-y^2, x^3, y^3)
\\
H^*(N, \Q) \, = \, 
H^*(P^2_\C \sharp \overline{P^2_\C}; \Q) \, &\simeq \,  \Q [x,y]/(x^2+y^2, x^3, y^3)
\endaligned
$$
is isomorphic to a subalgebra of the other.
\par

I am grateful to D.\ Kotschick for this example.
More generally, intersection forms play an important role 
for results and examples of domination for $4$-manifolds \cite{DuWa--04}.
\end{exes}

The following question has has attracted a lot of attention.

\begin{que}
Let $n$ be an integer, $n \ge 3$.
Let $M$ be an orientable connected closed $n$-manifold,
and $\mathcal N$ a class of such manifolds.
Decide whether the number of $N$ in $\mathcal N$ dominated by $M$ is finite, 
up to homotopy equivalence,
or up to homeomorphism when $n \le 3$.

\par
There is a natural variant of the question with $1$-domination.
\end{que}

As a small sample of known answers to this kind of questions, 
we quote the following results.

\begin{thm}
\label{thmBWZetc}
(1)
For $n \ge 2$, every hyperbolic oriented connected closed $n$-manifold
dominates a finite number only (up to homeomorphism) of such manifolds.
\par

(2)
Every oriented connected closed $3$-manifold
$1$-dominates a finite number only (up to homeomorphism)
of \emph{geometrizable} oriented connected closed $3$-manifolds.
\par

(3)
Every oriented connected closed $3$-manifold
dominates a finite number only (up to homeomorphism)
of irreducible non-geometrizable oriented connected closed $3$-manifolds.
\end{thm}

Claim (1) has already been observed for $n=2$ in Example \ref{exapresUmkehr}(3).
For $n \ge 4$, it follows from the fact that, for every $V > 0$, there are
up to homeomorphism (indeed up to isometry)
finitely many hyperbolic orientable connected closed manifolds of volume at most $V$
(a particular case of \cite[Theorem 8.1]{Wang--72}),
and from Gromov's work on simplicial volumes.
For $n=3$, we refer to \cite{Soma--98};
note that  \cite[Theorem 3.4]{BoWa--96}
has been corrected \cite[Remark 1.5]{BoRW--14}.
For (2) and (3), see respectively
\cite[Theorem 1.1]{Wang--02} and \cite[Theorem 1.1]{Liu}.

\begin{thm}[Boyer-Rolfsen-Wiest]
\label{thmRolfsen} 
Let $M,N$ be two oriented  connected closed $3$-manifolds.
Suppose that $M$ is prime, $\pi_1(N)$ left-orderable and non-trivial,
and $\pi_1(M)$ not left-orderable.
\par
Then $M$ does not dominate $N$.
\end{thm}

Recall that $M$ is prime if,
whenever $M$ is a connected sum $M_1 \sharp M_2$,
one of $M_1$, $M_2$ is a $3$-sphere.
By the Kneser-Milnor theorem, every oriented connected closed $3$-manifold
is a connected sum of prime $3$-manifolds, uniquely up to order and homeomorphism.
Theorem \ref{thmRolfsen} appears in
\cite[Theorem 1.1]{Rolf--04} and \cite[Theorem 3.7]{BoRW--05}.

\begin{thm}[\cite{Sun--15}]
\label{thmSun}
Let $M, N$ be oriented connected closed $3$-manifolds; 
suppose that $M$ is hyperbolic.
\par
Then there exists a finite covering $\widetilde M \longrightarrow M$
and a continuous map $\widetilde M \longrightarrow N$ of degree $2$.
In other words, $M$ virtually $2$-dominates 
every oriented connected closed $3$-manifold. 
\end{thm}

For ``most'' geometrizable $3$-manifolds $M$, 
the possible degrees of a continuous map from $M$ to itself are in $\{-1, 0, 1\}$.
Exceptions include manifolds covered by a torus bundle over the circle, 
or covered by $\Sigma_g \times S^1$ for $g \ge 2$,
or covered by one of $S^3$, $S^2 \times \R$.
Possible degrees of self-mappings of these ``exceptions'' 
are determined in \cite{SWWZ--12}.

\vskip.2cm

Attention has also been given to
domination between $3$-manifolds with boundaries;
as a sample, we cite only \cite{BBRW} and references there,
on domination of knot exteriors.

\vskip.2cm

To conclude this section, 
we describe some results of Gaifullin \cite{Gaif--08a, Gaif--08b, Gaif--13}
for manifolds of dimensions not necessarily $3$.
Claim (1) below involves the Tomei manifolds, defined as follows;
see \cite{Tome--84}, as well as \cite{Davi--87}.
\par
For $n \ge 1$, choose $n+1$ pairwise distinct real numbers,
say $\lambda_1, \hdots, \lambda_{n+1}$, 
and define the Tomei manifold $M_0^n$ as the space of
symmetric real matrices $(a_{i,j})_{1 \le i,j \le n+1}$ which are
tridiagonal, i.e.\ $a_{i,j} = 0$ whenever $\vert j-i \vert \ge 2$,
and with eigenvalues $\lambda_1, \hdots, \lambda_{n+1}$.
It can be shown that
\begin{enumerate}[noitemsep,label=(\roman*)]
\item\label{iDEtomeidavis}
$M_0^n$ is an orientable connected closed smooth $n$-manifold,
and its diffeomorphism type is independent of the choice of the $\lambda_i$ 's;
\item\label{iiDEtomeidavis}
$M_0^n$ is aspherical;
\item\label{iiiDEtomeidavis}
$\pi_1(M_0^n)$ is explicitely described in \cite{Davi--87} as a torsion-free finite-index
subgroup of a Coxeter group with $2n$ generators;
\item\label{ivDEtomeidavis}
$M_0^1$ is the circle, and $M_0^2$ is an orientable closed surface of genus $2$.
\end{enumerate}

\begin{thm}[Gaifullin]
\label{thmGaifullin}
(1)
For every $n \ge 2$, there exists an oriented connected closed $n$-manifold,
indeed the manifold $M_0^n$ described above,
such that every oriented connected closed $n$-manifold is dominated
by some finite cover of $M_0^n$.

(2)
For $n \in \{2,3,4\}$, 
there exists a hyperbolic oriented connected closed $n$-manifold 
of the form $H^n/\Gamma$
such that every oriented connected closed $n$-manifold is dominated
by some finite cover of $H^n/\Gamma$.
\end{thm}

Observe that Claim (1) is straightforward for $n=2$,
and true (but completely trivial) for $n=1$.
In (2), $H^n$ denotes the hyperbolic space of dimension $n$ 
and constant \mbox{curvature $-1$,}
and $\Gamma$ a torsion-free cocompact discrete subgroup of the isometry group of $H^n$.
As far as I know, it is unknown whether (2) extends to dimensions $n \ge 5$.
\vskip.2cm

Gaifullin considers also the more general problem of realising homology classes
by manifolds (see Remark \ref{SteenrodThom} below), 
but this goes beyond the scope of the present exposition.

\section{Simplicial volume}
\label{secsimvol}

Proposition \ref{propsimplicialvolmap} provides another obstruction to domination
of one manifold by another.
Simplicial volumes appear in \cite{Grom--82}; see also Chapter~6 of  \cite{Thur--80},
entitled ``Gromov's invariant and the volume of a hyperbolic manifold'', and \cite{Loh--11}.

\begin{defn}
\label{defsimplicialvol}
The \textbf{simplicial volume} $\Vert M \Vert_1$ of an oriented connected closed $n$-manifold $M$
is the $\ell^1$-semi-norm of the fundamental class $[M]$ in $H_n(M; \R)$,
i.e.\ the infimum of $\sum_k \vert h_k \vert$ over all singular cycles
$\sum_k h_k \sigma_k$ representing the fundamental class $[M] \in H_n(M; \R)$.
\par
More generally, for a topological space $X$ and an integer $j \ge 0$, 
every homology class $h \in H_j(X; \R)$ has a \textbf{$\ell^1$-semi-norm}
$\Vert h \Vert_1$ defined as the infimum of $\sum_k \vert h_k \vert$ over all singular cycles
$\sum_k h_k \sigma_k$ representing $h$.
\par
(Here, each $\sigma_k$ stands for a singular $j$-simplex, i.e.\  for 
a continuous map from the standard $j$-simplex
$\{ (x_0, \hdots, x_j) \in \R^{j+1} \mid x_0 \ge 0, \hdots, x_j \ge 0, \sum_{i=0}^j x_i = 1 \}$
to the space $X$).
\end{defn}

\begin{rem}
\label{boundedcoh}
Dually, a cohomology class $c \in H^j(X; \R)$ has a (possibly infinite!) $\ell^\infty$-semi-norm
$\Vert c \Vert_\infty$ defined as 
the infimum of $\Vert z \Vert_\infty$ over all singular cocycles $z$ 
representing $c$, and $\Vert z \Vert_\infty := \sup_\sigma \vert z(\sigma) \vert$,
where $\sup_\sigma$ stands for the supremum over the set of all singular simplices
$\sigma : \Delta^j \longrightarrow X$.
(It is natural to ask whether this semi-norm is a norm \cite[Page 38]{Grom--82}.
The answer is positive for $j \le 2$ \cite{MaMo--85}, but negative in general \cite{Soma--97}.)
\par

Let $M$ be as in the previous definition,
and $[M]^* \in H^n(M; \R)$ its dual fundamental class, as in \ref{dualFCZ}.
Then $\Vert [M]^* \Vert_\infty = \left( \Vert M \Vert_1 \right)^{-1}$;
in particular, $[M]^*$ is a bounded class if and only if $\Vert M \Vert_1 > 0$
\cite[Corollary on Page 17]{Grom--82}.
\end{rem}

The following basic and elementary proposition appears in \cite[Page 8]{Grom--82}.

\begin{prop}
\label{propsimplicialvolmap} 
Let $M,N$ be two oriented connected closed manifolds of the same dimension,
and $\varphi : M \longrightarrow N$ a continuous map. Then
$$
\Vert M \Vert_1  \, \ge \, 
\vert \deg   (\varphi) \vert  \hskip.1cm \Vert N \Vert_1  .
$$
If, moreover, $\varphi$ is a $d$-sheeted covering for some $d \ge 1$, 
then 
$$
\Vert M \Vert_1 \, = \,  d \Vert N \Vert_1 .
$$
\end{prop}

\begin{proof}
The first claim is a particular case of the following straightforward general fact:
if $f : X \longrightarrow Y$ is a continuous map between topological spaces
and $j$ a non-negative integer, then $\Vert f_j (c) \Vert_1 \le \Vert c \Vert_1$
for all $c \in H_j(X; \R)$.
\par

Set $n = \dim (N)$.
Let $\varepsilon > 0$. 
There exists a $n$-cycle $v = \sum_{\ell} v_\ell \tau_\ell$ representing $[N]$
such that $\Vert v \Vert_1 = \sum_\ell \vert v_\ell \vert \le \Vert N \Vert_1 + \varepsilon$.
Assume that $\varphi$ is a $d$-sheeted covering.
Since every $n$-simplex $\tau : \Delta^n \longrightarrow N$ has $d$ lifts
$\sigma : \Delta^n \longrightarrow M$, 
there exists a natural $n$-cycle $u = \sum_k u_k \sigma_k \in Z_n(M; \R)$
such that $\varphi_n (u) = dv$, and the class of $u$ is $[M]$.
We have therefore
$\Vert M \Vert_1 \le \Vert u \Vert_1 =  \Vert d v \Vert_1 \le 
\vert d \vert \Vert N \Vert_1 + \vert d \vert \varepsilon$,
and the second claim follows.
\end{proof}

Together with \ref{exessimplicialvol}(2), Proposition \ref{propsimplicialvolmap}  implies:

\begin{cor}
\label{0doesnotdominhyp}
If $\Vert M \Vert_1 = 0$ and $N$ is hyperbolic, then $M$ does not dominate $N$.
\end{cor}

\begin{exes}
\label{exessimplicialvol}
Let $M, N$ be oriented connected closed manifolds,
and $\Gamma$ the fundamental group of $M$.

\vskip.2cm

(1)
If there exists a continuous map $M \longrightarrow M$ of degree $d$
with $\vert d \vert \ge 2$, then $\Vert M \Vert_1 = 0$.
For example, $\Vert S^n \Vert_1 = 0$ for all $n \ge 1$, 
see Example \ref{exdeg}(1).
\par

For $k \ge 2$, the complex projective space $P^n_\C$ has
a continuous endomorphism of degree $k^n > 1$ given in homogeneous coordinates by
$$
(z_0 : z_1 : \cdots : z_n) \, \longmapsto \,  (z_0^k : z_1^k : \cdots : z_n^k) ;
$$
hence $\Vert P^n_\C \Vert_1  = 0$ for all $n \ge 1$.
Similarly, $\Vert P^n_\R \Vert_1 = 0$ for all odd $n \ge 3$
(recall that $P^n_\R$ is orientable for $n$ odd).
\par

If there exists a nontrivial circle action on $M$, then $\Vert M \Vert_1 = 0$
\cite[Page 41]{Grom--82}.
In particular, in dimension $3$,
if $M$ is a Seifert manifold, then $\Vert M \Vert_1 = 0$.

\vskip.2cm

(2)
If $M$ is hyperbolic, then $\Vert M \Vert_1 > 0$.
More precisely, if $M = H^n / \Gamma$,
where $H^n$ is the hyperbolic space of dimension $n$
(with $n \ge 2$),
and $\Gamma$ an appropriate cocompact discrete subgroup
of the group of orientation-preserving isometries of $H^n$, then
$$
\Vert M \Vert_1 \, = \, \operatorname{Vol} (M) / \nu_n ,
$$
where $\operatorname{Vol}$ indicates the Riemannian volume
and $\nu_n$ the maximal volume of an ideal $n$-simplex in $H^n$.
This is known as the Gromov-Thurston theorem; see e.g.\  \cite[Theorem 6.2]{Thur--80}.
\par

For surfaces, let $\Sigma_g$ denote an orientable connected closed surface of genus $g$,
and thus of Euler characteristic $\chi (\Sigma_g) = 2-2g$.
Then 
$$
\Vert \Sigma_g \Vert_1 \, = \,  2 \vert \chi (\Sigma_g) \vert \, = \, 2g-2 \, > \, 0
\hskip.5cm \text{when} \hskip.2cm g \ge 2
$$ 
and $\Vert \Sigma_g \Vert_1 = 0$ when $g = 0,1$.

\vskip.2cm

(3)
Consider a continuous map $f: \Sigma_g \longrightarrow \Sigma_h$,
and let $d = \deg (f)$ denote its degree.
We have already observed that, if $h = 0$, then $g$ and $d$ can be arbitrary
(Example \ref{exdeg}(1)).
\par

Assume from now on that $d \ne 0$.
By Example \ref{exapresUmkehr}(3), we know that $g \ge h$.
If $h=1$, then $g \ge 1$ and, similarly, $d$ can be arbitrary.
\par

Assume from now on that, moreover,  $h \ge 2$.
We have
$
2(g-1) \, = \, \vert \chi(\Sigma_g) \vert \,   \ge \,  \vert d \vert \hskip.1cm \vert \chi(\Sigma_h) \vert
\, = \,  \vert d \vert 2(h-1)
$ 
by Proposition \ref{propsimplicialvolmap}, i.e.\
$$
\vert d \vert \, \le \, \frac{g-1}{h-1}
$$
(a result already in Kneser \cite{Knes--30}).
This improves a conclusion in Example \ref{exapresUmkehr}(3).
We leave it to the reader to check that, 
for $g \ge h \ge 2$, 
every $d$ with $\vert d \vert  \le  \frac{g-1}{h-1}$
is the degree of some continuous map $\Sigma_g \longrightarrow \Sigma_h$.

\vskip.2cm

(4)
If $M$ is a compact Riemannian locally symmetric spaces of the non-compact type,
then $\Vert M \Vert_1 > 0$.
This has been conjectured in \cite[Page 11]{Grom--82},
and proved by the conjunction of  \cite{LaSc--06} and \cite{Buch--07}.
\par

More generally, Gromov asked if (or conjectured that ?)  this holds for
every manifold with non-positive curvature and negative Ricci curvature.

\vskip.2cm

(5)
For $M, N$ closed connected manifolds of dimension $m,n$ respectively:
\begin{equation}
\tag{$\Pi$}
\label{eqo3.5(5)}
\Vert M \Vert_1 \Vert N \Vert_1 \le \Vert M \times N \Vert_1 
\le \binom{m+n}{n} \Vert M \Vert_1 \Vert N \Vert_1 .
\end{equation}
\par

\emph{Note.}
Let $\Delta^m$ be a $m$-simplex in $\R^m$, with vertices $a_0, a_1, \hdots, a_m$,
and $\Delta^n$ a $n$-simplex in $\R^n$, with vertices $b_0, b_1, \hdots, b_n$.
In $\R^{m+n}$, there are $\binom{m+n}{m}$ sequences of the form
\footnotesize
$$
(a_0, b_0) = (a_{j_0}, b_{k_0}), (a_{j_1}, b_{k_1}), \hdots, (a_{j_\ell}, b_{k_\ell}) \hdots,
(a_{j_{m+n}}, b_{k_{m+n}}) = (a_m, b_n) ,
$$
\normalsize
where $(a_{j_\ell}, b_{k_\ell})$ is followed either by $(a_{j_\ell + 1}, b_{k_\ell})$
or by $(a_{j_\ell}, b_{k_\ell + 1})$, for $\ell \in \{0, 1, \hdots, m+n-1 \}$.
The  $(m+n)$-simplices of this form, $\binom{m+n}{m}$ of them,
constitute a triangulation of the product $\Delta^m \times \Delta^n$.
This is where the constant $\binom{m+n}{m}$ in (\ref{eqo3.5(5)}) comes from.

\vskip.2cm

(6)
Let $M$ be a closed connected manifold such that $\Vert M \Vert_1 > 0$.
In most of known examples \emph{which are not} hyperbolic manifolds,
the exact value of $\Vert M \Vert_1$ is not known.
\par

There is one remarkable exception, 
for manifolds covered by $H^2 \times H^2$.
For example, $\Vert \Sigma_g \times \Sigma_h \Vert_1 = 
\frac{3}{2} \Vert \Sigma_g \Vert_1  \Vert \Sigma_h \Vert_1$
\cite{Buch--08}.

\vskip.2cm

(7)
If $n \ge 3$ and $M = M_1 \sharp M_2$ is a connected sum
of oriented connected closed $n$-manifolds,
then $\Vert M \Vert_1 = \Vert M_1 \Vert_1 + \Vert M_2 \Vert_1$
\cite[Section 3.5]{Grom--82}.
Note that this is not true for surfaces ($n=2$) !

\vskip.2cm

(8)
Let $M$ be an oriented connected closed $3$-manifold; assume that $M$ is irreducible.
Let $H_1, \hdots, H_k, S_{k+1}, \hdots, S_\ell$ be the pieces of a JSJ-decomposition of $M$,
with $H_1, \hdots, H_k$ hyperbolic and $S_{k+1}, \hdots, S_\ell$ non-hyperbolic.
For $j = 1, \hdots, k$, denote by $\operatorname{Vol}(H_j)$ the volume of
the hyperbolic structure on the interior of $H_j$.
Then 
$$
\Vert M \Vert_1 \,  = \,  \frac{1}{\nu_3} \sum_{j=1}^k \operatorname{Vol}(H_j) ,
$$
with $\nu_3$ is the maximal volume of ideal $3$-simplices, as in (2) above.
See \cite{Soma--81}, 
and also \cite[Section 4.2, Corollary on Page 58]{Grom--82}.

\vskip.2cm

(9) 
If $\pi_1(M)$ is amenable, then $\Vert M \Vert_1 = 0$ .
In particular, if $M$ is simply connected or with finite fundamental group, 
then $\Vert M \Vert_1 = 0$.
See \cite[Section 3.1]{Grom--82}.

\vskip.2cm

(10)
Let $M$ be the total space of a fibre bundle
of which both the base space and the fibre are
oriented connected closed manifolds of positive dimensions.
If the fundamental group of the fibre is amenable, then $\Vert M \Vert_1 = 0$.
See \cite[Section 5.2]{Loh--11}.

\vskip.2cm

(11)
If a manifold $M$ of dimension at least $2$ is rationally essential 
(Definition \ref{defratess} below)
and if its fundamental group is Gromov-hyperbolic,
then $\Vert M \Vert_1 > 0$. See \cite[Corollary 5.3]{Loh--11}.

\vskip.2cm

(12)
If $n=4$, and $M$ fibers over an oriented base $B$ 
with fibers oriented surfaces $\Sigma_g$ of genus $g \ge 2$, then
$\Vert M \Vert_1 \ge \Vert \Sigma_g \Vert_1 \hskip.1cm \Vert B \Vert_1$
\cite{HoKo--01}.

\vskip.2cm

(13)
Consider the oriented  connected closed $3$-manifolds 
$M = S^1 \times \Sigma_g$ and $N = H^3 / \Gamma$,
where $\Sigma_g$ denotes an oriented connected closed surface of genus $g$ at least two,
$H^3$ the hyperbolic space of dimension $3$,
and $\Gamma$ an appropriate cocompact discrete subgroup
of the group of isometries of $H^3$.
\par

Since $\Vert M \Vert_1 = 0$ and $\Vert N \Vert_1 > 0$, 
the manifold $M$ does not dominate $N$, 
by Proposition  \ref{propsimplicialvolmap} or Example \ref{exdomin}(8).
\par

Observe however that, for $g$ at least as large as 
the minimal number of generators of $\Gamma$,
the obstructions of Proposition \ref{propdegmanif} do not apply.
\par

Indeed, there is a surjection of $\pi_1(\Sigma_g)$ onto the free group $F_g$ of rank $g$,
a fortiori a surjection of $\pi_1(M) = \Z \times \pi_1(\Sigma_g)$ onto $F_g$,
and a surjection of $F_g$ onto $\pi_1(N) = \Gamma$.
The composition $\pi_1(M) \longrightarrow \pi_1(N)$ is an epimorphism;
compare with Proposition \ref{propdegmanif}\ref{1DEpropdegmanif}. 
Also 
$$
H_1(M; \Q) \simeq \left(\pi_1(M) / [\pi_1(M), \pi_1(M)] \right) \otimes_\Z \Q
$$
surjects onto 
$$
H_1(N; \Q) \simeq \left(\pi_1(N) / [\pi_1(N), \pi_1(N)] \right) \otimes_\Z \Q
$$
and it follows that the Betti numbers satisfy
$\dim (H_j(M; \Q)) \ge \dim (H_j(N; \Q))$ for all $j$;
compare with Proposition \ref{propdegmanif}\ref{2DEpropdegmanif}.
\end{exes}

\begin{exe}
Here is an example of non-domination which can be proved by other methods,
giving rise to other kinds of obstructions.
\par
Let $M$ be a compact K\"ahler manifold,
and $N$ a compact locally symmetric space of the noncompact type
which is not locally hermitian symmetric, of the same dimension as $M$.
Then $M$ does not dominate $N$.
This is a particular case of results in \cite{CaTo--89},
obtained using Eells and Sampson's theory of harmonic mappings.

For example, if $\Gamma$ is a torsion-free cocompact lattice in $\operatorname{SL}_n(\R)$,
with $n > 2$, a compact K\"ahler manifold does not dominate
$\Gamma \backslash \operatorname{SL}_n(\R) / \operatorname{SO}(n)$.
\end{exe}

\begin{rem}
There is a ``mapping theorem'' in bounded cohomology \cite[Section 3.1]{Grom--82}:
Let $X, Y$ be connected topological spaces which are good enough 
(say CW-complexes, for simplicity)
and $f : X \longrightarrow Y$ a continuous map which induces an isomorphism
$f_\pi = \pi_1(X) \longrightarrow \pi_1(Y)$. Then $f^* : H^*_b(Y) \longrightarrow H^*_b(X)$
is an isometric isomorphism (with respect to the norms $\Vert \cdot \Vert_\infty$).
\par

It follows that, if $M$ is an oriented connected closed manifold
and $c : M \longrightarrow B(\pi_1(M))$ its classifying map, 
then $\Vert M \Vert_1 = \Vert f_n ([M]) \Vert_1$.
\par

Note however that two manifolds with isomorphic fundamental groups
\emph{need not} have the same simplicial volume.
For example, if $M$ is such that $\Vert M \Vert_1 > 0$,
then $\Vert M \times S^2 \Vert_1 = 0$
by Example \ref{exessimplicialvol}(5);
of course $\pi_1(M \times S^2) \simeq \pi_1(M)$.
\end{rem}

\section{Domination by products for manifolds and for groups}
\label{sectionproduct}

Here is a particular case of Question \ref{questionGrom}.

\begin{que}
\label{questionGromProd}
Let $N$ be an oriented connected closed manifold of dimension $n$. 
Can $N$ be dominated by a product ? 
\par
In other words, do there exist two
oriented connected closed manifolds $M_1, M_2$, 
say of dimensions $n_1, n_2 \ge 1$ with  $n_1 + n_2 = n$, 
such that $M_1 \times M_2$ dominates $N$ ?
\end{que}

Proposition \ref{propKL09} below shows that, in some cases,
Question \ref{questionGromProd} can be translated to fundamental groups.

\begin{rem}
\label{SteenrodThom}
Question \ref{questionGromProd} refers more or less implicitely to other very classical questions.
\par
Steenrod's problem on realization of cycles asks:
\emph{for a finite polyhedron $K$ and a homology class $z \in H_j(K; \Z)$,
does there exist an oriented connected closed manifold $M$ of dimension $j$ 
and a map $f : M \longrightarrow K$
such that $z$ is the image by $f_j$ of the fundamental class of $M$~?} 
see \cite[Problem 25]{Eile--49}.
The answer is positive when $j \le 5$;
it need not be positive in the general case 
\cite[th\'eor\`emes III.3 and III.9]{Thom--54}, 
but that there exists
always a multiple $kz$ (where $k \in \Z$ depends on $j$ only) 
which is of the form $f_j([M])$  \cite[th\'eor\`eme III.4]{Thom--54}.
\par

There is in \cite[Page 304]{Grom--99} a discussion of classes in $H_j(\cdot; \Z)$, with $j$ even,
which can be represented by products of surfaces.
\end{rem}

\begin{reminder}[Classifying spaces of coverings]
Let us review some material on classifying spaces for coverings.
To my surprise, the only references I found on this
are in the setting of topological groups and principal bundles
(see e.g. \cite{Dold--63}),
rather than in the simpler setting of groups and coverings.
Say here that a topological space is ``good enough'' if it is
path connected, locally path connected, and locally simply connected
(so that fundamental group and coverings are simply dealt with).

\vskip.2cm

Let $\Gamma$ be a group.
Recall that there exists a \textbf{classifying covering} 
$\pi_\Gamma : E(\Gamma) \longrightarrow B(\Gamma)$
where $E(\Gamma), B(\Gamma)$ are topological spaces
and $B(\Gamma)$ is a pointed space, with the following property.
For any space $X$ which is good enough and paracompact
and for any covering $p : E \longrightarrow X$ of group $\Gamma$,
there exists a continuous map (well-defined up to homotopy equivalence) 
$c : X \longrightarrow B(\Gamma)$ such that the covering $p$ is isomorphic to 
the pullback by $c$ of the classifying covering $\pi_\Gamma$.
The base space $B(\Gamma)$ is the \textbf{classifying space for the group $\Gamma$}.
\par

It is a basic result that a covering of group $\Gamma$ over a good enough space
is classifying if and only if its total space is contractible.
Consequently, another name for ``classifying space for $\Gamma$'' 
is ``Eilenberg-MacLane space $K(\Gamma,1)$''.
\par

Let $\Delta$ be another group 
and $\varphi : \Gamma \longrightarrow \Delta$ a homomorphism.
Given a covering $p : E \longrightarrow X$ of group $\Gamma$,
there is an induced covering $(\varphi)_*p : F \longrightarrow X$ of group $\Delta$, 
in which the total space $F$ is the quotient $E \times_\Gamma \Delta$ of $E \times \Delta$
by the equivalence relation generated by 
$(\xi,\delta) \sim (\xi \gamma, \varphi(\gamma)\delta)$,
for all $\xi \in E, \delta \in \Delta, \gamma \in \Gamma$.
\par

The assignment to a group of its classifying space is functorial in the following sense:
to every homomorphism $\varphi : \Gamma \longrightarrow \Delta$,
corresponds a continuous map $B(\varphi) : B(\Gamma) \longrightarrow B(\Delta)$
mapping base point to base point, 
such that, if $p : E \longrightarrow X$ is a covering of group $\Gamma$,
pulled back as above from $\pi_\Gamma$ 
by some continuous map $c : X \longrightarrow B(\Gamma)$,
then $(\varphi)_*p$ is isomorphic to the covering of group $\Delta$
pulled back from $\pi_\Delta$ by the composition $B(\varphi) \circ c$.
\par

One way to prove these claims about the functor $B(\cdot)$
is to invoke Milnor's join construction of $B(\Gamma)$ \cite{Miln--56}.

For example, let $\varphi$ be the reduction modulo $2$ 
from $\Gamma = \Z$ to $\Delta = \Z / 2\Z$.
The standard model for the classifying space of $\Z$ is the circle $\R / \Z$
(the total space of the classifying covering is $\R$),
and that for $\Z / 2\Z$ is the infinite real projective space $P^\infty_\R = \lim_n P^n_\R$
(the corresponding total space is an inductive limit $\lim_n S^n$ of spheres).
Then $B(\varphi) : \R / \Z \longrightarrow P^\infty_\R$ maps the circle
to a loop which represents the non-trivial element of the fundamental group of $P^\infty_\R$.

Let $\Gamma = \Gamma_1 \times \Gamma_2$ be a group 
given as a direct product of two subgroups.
The two projections $\varphi_j : \Gamma \longrightarrow \Gamma_j$
provide a map $B(\Gamma) \longrightarrow B(\Gamma_1) \times B(\Gamma_2)$
\emph{which is an isomorphism}.
Indeed, since homotopy groups of a product of spaces 
are products of homotopy groups,
we have $\pi_1 \left( B(\Gamma_1) \times B(\Gamma_2) \right) = \Gamma_1 \times \Gamma_2 = \Gamma$
and $\pi_j (\text{idem}) = \{0\}$ for $j \ge 2$;
hence $B(\Gamma_1) \times B(\Gamma_2)$ is a classifying space for $\Gamma$, as claimed.
\par

Consider in particular two connected closed manifold $M, N$ and a continuous map
$f : M \longrightarrow N$.
Let $\widetilde M$ denote the universal covering of $M$, which is a covering of group $\pi_1(M)$,
pullback of the classifying covering for this group 
by some continuous map $c_M : M \longrightarrow B(\pi_1(M))$.
Similarly, the universal covering $\widetilde N$ 
is pulled back of the classifying covering for $\pi_1(N)$
by some continuous map $c_N : N \longrightarrow B(\pi_1(N))$.
\par
\begin{center}
\emph{Then the maps $c_N \circ f$ and $B(f_\pi) \circ c_M$ are homotopic.} 
\end{center}
\end{reminder}

The following definition is a variation on a notion introduced in 
\cite[Page 95]{Grom--82} and \cite[$\S$~0]{Grom--83}.
For a group $\Gamma$, we write $H_*(\Gamma; \Q)$ for $H_*(B(\Gamma); \Q)$.

\begin{defn}
\label{defratess}
An orientable connected closed manifold $M$ of positive dimension $n$ is
\textbf{rationally essential} if, with the notation above,
the image of a fundamental class $[M]$ 
by the map $(c_M)_n : H_n(M; \Q) \longrightarrow H_n(\pi_1(M); \Q)$
 is not zero.
\end{defn}

\begin{exes}
Let $M$ be an orientable connected closed manifold.

\vskip.2cm

(1)
A connected manifold is \textbf{aspherical} if its universal covering is contractible;
more on these manifolds in \cite{Luck--12}.
If $M$ is aspherical, 
then $M$ is rationally essential, and indeed $M$ itself is a good model for $B(\pi_1(M))$.
There are many aspherical examples of orientable connected closed manifolds.
The most standard examples include:

   (1.i) all surfaces, except the sphere;

   (1.ii) all $3$-manifolds except those of the three following classes
(as a consequence of the sphere theorem of Papakyriakopoulos):
manifolds finitely covered by the $3$-sphere, 
non-trivial connected sums (they have non-trivial $\pi_2$), 
and $S^1 \times S^2$;

   (1.iii) all manifolds admitting a Riemannian metric with nonpositive sectional curvature
(a consequence of the Cartan-Hadamard theorem);
   
   (1.iv) all manifolds of the form $\Gamma \backslash G / K$,
with $G$ a connected Lie group, 
$K$ a maximal compact subgroup of $G$,
and $\Gamma$ a torsion-free cocompact lattice in $G$
(a consequence of the Cartan-Malcev-Iwasawa theorem);
when $G$ is linear semi-simple without compact factors, 
the corresponding double coset space is an example of (1.iii).

\vskip.2cm

(2) 
If the simplicial volume $\Vert M \Vert_1$ is positive, 
then $M$ is rationally essential  \cite[Section 3.1]{Grom--82}.

\vskip.2cm

(3)
If there exists a rationally essential manifold $N$ and a continuous map
$M \longrightarrow N$ of non-zero degree, then $M$ is rationally essential.
For example, every connected sum $M_1 \sharp M_2$ with $M_2$ rationally essential
is rationally essential.

\vskip.2cm

(4)
If $\pi_1(M)$ is finite, then $M$ is not rationally essential,
since $H_j(\pi_1(M); \Q) = \{0\}$ for all $j \ge 1$.
\end{exes}

\begin{prop}[Kotschick-L\"oh \cite{KoLo--09}]
\label{propKL09}
Let $M, N$ be two oriented connected closed manifolds 
and $f : M \longrightarrow N$ a continuous map; set $\Gamma = \pi_1(M)$.
Assume that $M = M_1 \times M_2$ is a non-trivial product.
For $j=1,2$, 
\begin{enumerate}
\item[$\cdot$]
denote by $n_j \ge 1$ the dimension of $M_j$; 
\item[$\cdot$]
choose a base point $x_j \in M_j$; 
\item[$\cdot$]
set $\Gamma_j = \pi_1(M_j, x_j)$, so that $\Gamma = \Gamma_1 \times \Gamma_2$;  
\item[$\cdot$]
denote by $f_j : M_j \longrightarrow N$ ($j=1,2$) the restriction of $f$
to $M_1 \times \{x_2\}$, $\{x_1\} \times M_2$ respectively; 
\item[$\cdot$]
denote by $\Delta_j$ the image $(f_j)_\pi(\Gamma_j)$.
\end{enumerate}
Observe that $\Delta_1$ and $\Delta_2$ are commuting subgroups of $\pi_1(N)$.
\begin{enumerate}
\item[$\cdot$]
Denote by $\psi$ the multiplication homomorphism 
$$
\Delta_1 \times \Delta_2 \,  \longrightarrow \,  \pi_1(N), 
\hskip.2cm (\delta_1, \delta_2) \longmapsto \delta_1\delta_2 ,
$$
so that the image of $\psi$ coincides with the image of $f_\pi$.
\end{enumerate}
Suppose moreover that $N$ is rationally essential.
\vskip.1cm

If the degree of $f$ is not zero, then both $\Delta_1$ and $\Delta_2$ are \emph{infinite}
subgroups of $\pi_1(N)$.
\end{prop}

\begin{proof}
Recall that $f_\pi$ denotes the homomorphism $\pi_1(M) \longrightarrow \pi_1(N)$ induced by $f$.
We have a diagram
\begin{equation*}
\dia{
M_1 \times M_2 = M \hskip.5cm
\ar@{->}[dd]^{f}
\ar@{->}[r]^{c_{M_1} \times c_{M_2} = c_M}
& \hskip.5cm B(\Gamma_1) \times B(\Gamma_2) \hskip.5cm
\ar@{->}[d]^{B\left((f_1)_\pi\right) \times B\left((f_2)_\pi\right)}
\ar@{->}[r]^{\hskip1cm =}  
& \hskip.5cm B(\Gamma)
\ar@{->}[dd]^{B\left( f_\pi \right)}
\\
& \hskip.5cm B(\Delta_1) \times B(\Delta_2)
\ar@{->}[d]^{B(\psi)}
&
\\
N
\ar@{->}[r]^{c_N}
& B(\pi_1(N))
\ar@{->}[r]^=
&  \hskip.2cm B(\pi_1(N))
}
\end{equation*}
commuting up to homotopy.
(We have abusively denoted by ``$=$'' several canonical isomorphisms.)
\par
There is a corresponding commutative diagram for the homologies $H_n(\cdot; \Q)$,
in which there are three ways to go from the top left corner to the bottom right corner.
\par

(i) Going through the left-hand vertical arrow, 
the fundamental class $[M]$ in $H_n(M; \Q)$ 
is mapped to $\deg (f) (c_N)_n([M])$ in $H_n(\pi_1(N); \Q)$.
\par

(ii) Going through the central vertical arrow, 
$[M]$ goes through a class 
$h_1 \otimes h_2 \in H_{n_1}(\Delta_1; \Q) \otimes H_{n_2}(\Delta_2; \Q)
\subset H_n(\Delta_1 \times \Delta_2; \Q)$.
\par

(iii)
Going  through the right-hand vertical arrow, $[M]$ goes to
$$
\left( B(f_\pi) \right)_n \circ \left( c_M \right)_n ([M]) 
\, \in \, \left( B(f_\pi) \right)_n \left( H_n(\Gamma; \Q) \right)
\, \subset \, H_n(\pi_1(N); \Q) .
$$
These three images of $[M]$ in $H_n(\pi_1(N); \Q)$ coincide.
\par

If $\deg (f) \ne 0$ and if $N$ is rationally essential,
this common image is not zero.
Hence $h_1 \otimes h_2 \ne 0$;
in particular 
$H_{n_1}(\Delta_1; \Q) \ne \{0\}$ and $H_{n_2}(\Delta_2; \Q) \ne \{0\}$.
It follows that $\Delta_1$ and $\Delta_2$ are infinite groups.
\end{proof}

This suggests the following definition, again from \cite{KoLo--09}:

\begin{defn}
\label{defpbp}
A group $\Delta$ is \textbf{presentable by a product},
or shortly below \textbf{pbp}, 
if there exist two groups $\Gamma_1, \Gamma_2$ and a homomorphism
$\varphi : \Gamma_1 \times \Gamma_2 \longrightarrow \Delta$
with the two following properties:

(i) 
both $\varphi(\Gamma_1)$ and $\varphi(\Gamma_2)$ are infinite subgroups of $\Delta$,

(ii) 
these subgroups generate in $\Delta$ a subgroup of finite index.
\end{defn}

\begin{exes}
(1)
A group $\Gamma$ with infinite centre is pbp, 
because the multiplication
$\Gamma \times Z(\Gamma) \longrightarrow \Gamma$
has Properties (i) and (ii) above.

\vskip.2cm

(2)
Let $\Delta$ be a group and $\Delta'$ a subgroup of finite index.
Then $\Delta$ is pbp if and only if $\Delta'$ is pbp.

\vskip.2cm

(3)
Let $\Delta = \Delta_1 \ast \Delta_2$ be a free product of two non-trivial groups.
Then $\Delta$ is pbp if and only if both $\Delta_1$ and $\Delta_2$ are of order $2$
(and therefore $\Delta$ is an infinite dihedral group).
See \cite[Corollary 9.2]{KoLo--13}.

\vskip.2cm

(4)
An infinite simple group is not pbp.

\vskip.2cm

(5)
A non-elementary Gromov hyperbolic group is not pbp
\cite[Theorem 1.5]{KoLo--09}.

\vskip.2cm

(6)
For a closed $3$-manifold $M$ with infinite fundamental group, the following properties
are equivalent \cite[Theorem 8]{KoNe--13}:

(i)
$\pi_1(M)$ is pbp,

(ii)
$\pi_1(M)$ has a finite index subgroup with infinite centre,

(iii)
$M$ is a Seifert manifold.

\vskip.2cm

\noindent
For other examples, see 
\cite{KoLo--09, KoLo--13, HaKo},
as well as Section \ref{sectionCox} below.
\end{exes}

With the terminology of Definition \ref{defpbp},
Proposition \ref{propKL09} can be reformulated as follows
\cite[Theorem 1.4]{KoLo--09}:

\begin{thm}[Kotschick-L\"oh]
\label{propKL09bis}
Let $N$ be an oriented connected closed manifold.
Assume that $N$ is rationally essential and that $\pi_1(N)$ is not pbp.
Then $N$ is not dominated by any non-trivial product
of closed oriented connected manifolds.
\end{thm}

In the converse direction, let us quote the following result \cite[Theorem A]{Neof}:

\begin{thm}[Neofytidis]
Let $N$ be an oriented connected closed manifold.
Assume that $N$ is aspherical and that there is a subgroup of finite index in $\pi_1(N)$
that is a direct product of two infinite groups.
Then $N$ is dominated by a non-trivial direct product of oriented connected closed manifolds.
\end{thm}

In contrast, there exist oriented connected closed $4$-manifolds $M$ 
such that $\pi_1(M)$ is the direct product of two infinite groups,
and yet $M$ is not dominated by any non-trivial product
\cite[Exampe 6.3]{KoLo--09}. Of course, these examples cannot be aspherical.

\section{Some examples, including Coxeter groups}
\label{sectionCox}

In this section, $k$ denotes a field of characteristic $0$.
For more details, see \cite{HaKo}.
\par
Definition \ref{defpbp} has analogues in other categories, e.g.:

\begin{defn}
\label{defpbpLiealg}
A $k$-Lie algebra $\mathfrak g$ is \textbf{presentable by a product} 
if there exist two commuting subalgebras
$\mathfrak g_1, \mathfrak g_2$ of $\mathfrak g$ of positive dimensions
such that the homomorphism 
$\mathfrak g_1 \times \mathfrak g_2 \longrightarrow \mathfrak g$,
$(X_1 , X_2) \longmapsto X_1 + X_2$, is surjective.
\par

Note that, if $\mathfrak g_1, \mathfrak g_2$ are as above, 
then they are non-zero ideals in $\mathfrak g$,
and each one is contained in the centralizer in $\mathfrak g$ of the other.
\end{defn}

\begin{defn}
\label{defpbpalggps}
A connected linear algebraic $k$-group $\G$
is \textbf{presentable by a product} if there exist
two commuting connected closed $k$-sub\-groups 
$\G_1, \G_2$ of $\G$ of positive dimensions
such that the multiplication homomorphism
$\mu : \G_1 \times \G_2 \longrightarrow \G$, $(g_1, g_2) \longmapsto g_1g_2$,
is surjective.

Note that, if $\G_1, \G_2$ are as above,
then they are non-trivial normal subgroups of $\G$, 
and each one is contained in the centralizer of the other.
\end{defn}

\begin{prop}
\label{changecategories}
Let $\G$ be a connected linear algebraic $k$-group, $\mathfrak g$ its Lie algebra,
and $\Gamma$ a Zariski dense subgroup of the group $G = \G (k)$ of rational points of $\G$.
\par
(1) If $\mathfrak g$ is not presentable by a product, then $\G$ is not presentable by a product.
\par
(2) If $\G$ is not presentable by a product, then $\Gamma$ is not presentable by a product.
\par\noindent
$[$\emph{The proof is routine; we refer to \cite{HaKo}.}$]$
\end{prop}

\begin{exes}
\label{expbpGammainG}
(1)
Let 
$\operatorname{Af}_1(\R) = 
\begin{pmatrix}
\mathbf R^\times & \mathbf R 
\\
0 & 1 
\end{pmatrix}$
be the real affine group,
or more precisely the group of real points of the appropriate algebraic group.
Observe that $\operatorname{Af}_1(\R)$ is centreless, of dimension $2$.
As an algebraic group, it is connected, even if, as a real Lie group, 
its connected component
$ 
\begin{pmatrix}
\mathbf R_+^\times & \mathbf R 
\\
0 & 1 
\end{pmatrix}$
is of index two.
\par
Its Lie algebra is the non-abelian soluble Lie algebra 
$\mathfrak{af}_1(\R)$ of dimension $2$, 
with basis $\{e,g\}$ such that $[g,e] = e$.
\par

For every $n \in \Z$ with $\vert n \vert \ge 2$, the soluble Baumslag-Solitar group
$\operatorname{BS}(1,n)$,
of presentation
$\langle a,t \mid tat^{-1} = a^n \rangle$,
is isomorphic to a Zariski dense subgroup of $\operatorname{Af}_1(\R)$:
\begin{equation*}
\operatorname{BS}(1,n) \, \approx \, 
\begin{pmatrix}
\mathbf n^{\mathbf Z} & \mathbf Z [1/n] 
\\
0 & 1 
\end{pmatrix} \, \subset \,
\begin{pmatrix}
\mathbf R^\times & \mathbf R 
\\
0 & 1 
\end{pmatrix}
\, = \,
\operatorname{Af}_1(\R) .
\end{equation*}

Since $\mathfrak{af}_1(\R)$ has a unique non-trivial ideal, $\R e$ of dimension one, 
$\mathfrak{af}_1(\R)$ is not presentable by a product.
It follows  from Proposition \ref{changecategories}
that the group $\operatorname{BS}(1,n)$ is not presentable by a product.
\par

It is a fact that all Baumslag-Solitar groups $\operatorname{BS}(m,n)$
with $\vert m \vert \ne \vert n \vert$ are not presentable by products \cite{HaKo}. 

\vskip.2cm

(2) 
For $n \ge 2$, the group $\operatorname{SL}_n(\Z)$
is not presentable by a product.
More generally, every Zariski dense subgroup of $\operatorname{SL}_n(\R)$
is not presentable by a product.

\vskip.2cm

(3)
Consider a semi-direct product $\mathfrak h = V^r \rtimes \mathfrak{g}$,
where $V$ is a finite dimensional $k$-vector space,
$\mathfrak g$ a simple Lie subalgebra of $\mathfrak{gl}(V)$
acting irreducibly on $V$, and $r$ a positive integer;
the space  $V^r = V \oplus \cdots \oplus V$ (with $r$ factors) 
is considered as an abelian Lie algebra
on which $\mathfrak g$ acts by the restriction of the natural diagonal action
of  $\mathfrak{gl}(V)$ on $V^r$.
Then $\mathfrak h$ is not presentable by a product
(Hint:
every non-zero ideal of $V^r \rtimes \mathfrak g$
is either the full Lie algebra, with centre $\{0\}$, 
or a $ \mathfrak g$-invariant subspace of $V^r$, with centralizer $V^r$.)
\end{exes}

Consider now a Coxeter system $(W,S)$, with $S$ finite.
Let $B_S$ denote the Tits form on the vector space $\R^S$,
and $\operatorname{Of}(B_S)$ the group of invertible linear transformations
$g \in \operatorname{GL}(\R^S)$ such that 
$B_S(gv,gw) = B_S(v,w)$ for all $v,w \in \R^S$ and $gv = v$ for all $v \in \ker (B_S)$.
\par
Suppose that $(W,S)$ is irreducible, and neither finite nor affine.
On the one hand,
it is easy to check that its Lie algebra
$\mathfrak{of}(B_S)$ is of the form of $\mathfrak h$ in Example \ref{expbpGammainG}(3),
so that it is not presentable by a product.
On the other hand, the geometric representation 
$\sigma_S : W \longrightarrow \operatorname{Of}(B_S)$ is Zariski-dense \cite{BeHa--04}.
It follows that $W$ is not presentable by a product.
To summarize:

\begin{exe}
Let $(W,S)$ be a Coxeter system, with $S$ finite, which is irreducible, infinite, and non-affine.
Then $W$ is not presentable by a product.
\end{exe}

\end{document}